\documentclass[11pt]{amsart}
\usepackage{amsmath,amssymb,amscd,mathrsfs}
\usepackage{url}
\usepackage{eepic,epic}
\usepackage{longtable}
\usepackage{hyperref}
\usepackage[noend]{algpseudocode}
\numberwithin{equation}{section} 

\usepackage{enumerate} 
\usepackage{color} 

\newtheorem{theorem}{Theorem}[section]
\newtheorem{proposition}[theorem]{Proposition}
\newtheorem{lemma}[theorem]{Lemma}
\newtheorem{corollary}[theorem]{Corollary}
\theoremstyle{definition}

\newtheorem{remark}[theorem]{Remark}

\newtheorem{question}[theorem]{Question}

\newcommand{\EnS}{Y}
\newcommand{\KtS}{X}

\newcommand{\tdfn}{f_N}
\newcommand{\tdfm}{f_M}

\def\Aut{\mathop{\mathrm{Aut}}\nolimits}

\def\C{\mathbb C}  \def\R{\mathbb R} \def\Z{\mathbb Z} \def\N{\mathbb N}

\newcommand{\OG}{\mathord{\mathrm{O}}}
\newcommand{\intf}[1]{\langle #1 \rangle}

\DeclareMathOperator{\disc}{det}

\DeclareMathOperator{\rank}{rank}

\DeclareMathOperator{\res}{res}
\DeclareMathOperator{\tr}{Tr}
\newcommand{\divides}{\mid}

\newenvironment{(enumerate)}{
  \begin{enumerate}
  
  }{\end{enumerate}}
\newenvironment{anumerate}{
  \begin{enumerate}
  
  }{\end{enumerate}}

\newcommand{\even}{\mathrm{II}}

\DeclareMathOperator{\NS}{NS}
\DeclareMathOperator{\GL}{GL}

\DeclareMathOperator{\End}{End}
\newcommand{\QQ}{\mathbb{Q}}
\newcommand{\FF}{\mathbb{F}}
\newcommand{\NN}{\mathbb{N}}
\newcommand{\RR}{\mathbb{R}}

\newcommand{\ZZ}{\mathbb{Z}}

\newcommand{\NumS}{\operatorname{Num}(Y)}

\begin{document}
\title[On characteristic polynomials of automorphisms]{On characteristic polynomials of automorphisms of Enriques surfaces}
\author[S. Brandhorst]{Simon Brandhorst}
\address{Fachbereich Mathematik, Saarland University,
Campus E2.4 Zi. 222, 66123 Saarbr\"ucken, Germany}
\email{brandhorst@math.uni-sb.de}

\author[S. Rams]{S{\L}awomir Rams}
\address{Institute of Mathematics, Jagiellonian University,
ul. {\L}ojasiewicza 6,  30-348 Krak\'ow, Poland}
\email{slawomir.rams@uj.edu.pl}

\author[I. Shimada]{Ichiro Shimada}
\address{Department of Mathematics,
Graduate School of Science,
Hiroshima University,
1-3-1 Kagamiyama,
Higashi-Hiroshima,
739-8526 JAPAN}
\email{ichiro-shimada@hiroshima-u.ac.jp}

\thanks{S.~B. is supported by  SFB-TRR 195 ”Symbolic Tools in Mathematics and their Application” of the German Research Foundation(DFG).
S.~R. is partially supported by the Polish National Science Centre (NCN) OPUS grant  2017/25/B/ST1/00853.
I.~S.  is supported by JSPS KAKENHI Grant Number 15H05738, ~16H03926,  and~16K13749}

\begin{abstract}
Let $f$ be an automorphism of a complex  Enriques surface $\EnS$
and let $p_f$ denote the characteristic polynomial
of the isometry $f^*$ of the numerical N\'eron-Severi lattice of $\EnS$
induced by $f$.
We apply a modification of McMullen's method to prove
that the modulo-$2$ reduction $(p_f(x)  \bmod 2)$ is a product of
modulo-$2$ reductions of (some of) the five cyclotomic polynomials
$\Phi_m$, where $m \leq 9$ and $m$ is  odd. We study Enriques surfaces that realize  
modulo-$2$ reductions of $\Phi_7$, $\Phi_9$
 and show that each of the five polynomials  $(\Phi_m(x)  \bmod 2)$
is a factor of the modulo-$2$ reduction $(p_f(x)  \bmod 2)$ for a complex  Enriques surface.
\end{abstract}
\date{23rd of September 2019}
\maketitle

\section{Introduction}
The subject of this note are isometries 
of the numerical N\'eron-Severi lattices 
induced by automorphisms of Enriques surfaces.
To state our results, let $\EnS$ (resp. $\KtS$) be a complex Enriques surface (resp. its K3 cover) and let $\NumS$ 
be the numerical N\'eron-Severi lattice of $\EnS$ (i.e. $\NumS := \NS(\EnS)/\mbox{Tors}$). 
Each automorphism $f \in \Aut(\EnS)$ induces an isometry   $f^* \in \OG(\NumS)$.
It is natural to study the  properties of the characteristic polynomial of the latter.

In this note we prove the following refinement of \cite[Theorem 1.2]{hor}.
\begin{theorem} \label{main}
Let $f$ be an automorphism of a complex Enriques surface $\EnS$ and let $p_f$ be the characteristic polynomial of the isometry $f^*: \NumS \rightarrow \NumS$.

a) The modulo-$2$ reduction $(p_f(x)  \bmod 2)$  is a product of (some of) the following polynomials:
\begin{align*}
F_{ 1}(x) &=                                                          x + 1, \quad
F_{ 3}(x)  =                                                    x^2 + x + 1, \quad
F_{ 5}(x)  =                                        x^4 + x^3 + x^2 + x + 1, \\
F_{ 7}(x) &=                            x^6 + x^5 + x^4 + x^3 + x^2 + x + 1, \quad
F_{ 9}(x) =                            x^6 +             x^3           + 1 \, .
\end{align*}

b) Each of the five polynomials $F_{1}, F_{3}, F_5, F_7, F_9$
does appear in the factorization of the modulo-$2$ reduction $(p_f(x)  \bmod 2)$ for an automorphism $f$ of a complex Enriques surface. Any realization of $F_7$ and $F_9$ is by a semi-symplectic automorphism.
\end{theorem}
Recall that the proof of \cite[Theorem 1.2]{hor} shows that each factor  of $(p_f(x)  \bmod 2)$ either equals one of the five polynomials listed in Thm~\ref{main}, or it is
the modulo-$2$ reduction  $F_{15}$ of the cyclotomic polynomial $ \Phi_{15}\in \mathbb{Z}[x]$.
Moreover, examples with factors $F_1$, $F_3$, $F_5$ were given in \cite{dolgachev16} (see also \cite[Example 3.1]{hor}), whereas the question  if  $F_{ 7}$, $F_{ 9}$ and $F_{15}$ can appear in the factorization
of the modulo-$2$ reduction of $p_f$ for an automorphism $f \in \Aut(\EnS)$  was left open (c.f. \cite[Example 3.1.b]{hor}).
This question is answered in Theorem \ref{main}.

To state the next theorem, we introduce some notation.
Let us denote the covering involution of the double \'etale cover $\pi: \KtS \rightarrow \EnS$ by  $\varepsilon$.
Moreover, we put  $\tilde{f} \in \Aut(\KtS)$ to denote a (non-unique) lift of
an automorphism  $f \in \Aut(\EnS)$. Let $N := (H^2(\KtS,\Z)^{\varepsilon})^{\perp}$ be
the  orthogonal complement of the $\varepsilon$-invariant sublattice $H^2(\KtS,\Z)^{\varepsilon}$ in the 
lattice $H^2(\KtS,\Z)$. Recall that $N$    is stable under the cohomological action $\tilde{f}^*$ and  the restriction 
$\tdfn := \tilde f^* \rvert_N$
is of finite order. Using Theorem \ref{main}, we can sharpen \cite[Theorem 1.1]{hor} as well.
\begin{theorem}\label{thm-orders}
 Let $\EnS$ be a complex Enriques surface and let 
 $f$ be an automorphism of $\EnS$.
 Then, the order of $f_N$ is
 a divisor of at least one of the following five integers:
\[36, 48, 56, 84, 120.\]
Among the 28 numbers that satisfy the above condition,  
at least the following $14$ integers
\[1,\dots , 10, 12, 14, 15, 20\]
are realized as orders.
\end{theorem}
\begin{remark}
We note that if the order of $f_N$ is $7$ or $9$, then the cyclic subgroup generated by $f_N$ is unique up to conjugacy in the orthogonal group $\OG(N)$.
For the remaining $14$ integers
\[16, 18, 21, 24, 28, 30, 36, 40, 42, 48, 56, 60, 84, 120,\]
we do not know whether they arise as orders of $f_N$ for some $f \in \Aut(\EnS)$.
\end{remark}
Originally, our interest in the subject of this note was motivated by the question what constraints on the dynamical spectra of Enriques surfaces result from the existence
of the double \'etale K3 cover (c.f. \cite[Theorem 1.2]{third}). Indeed, Theorem~\ref{main}.a yields a new constraint on
the Salem numbers that appear as
the dynamical degrees of automorphisms of Enriques surfaces (e.g. it implies that none of the  Salem numbers given as $\#$~3,~13,~16,~34,~35 in the table
in \cite[Appendix]{hor} can be the dynamical degree of an automorphisms of a complex Enriques surface), whereas  Theorem~\ref{main}.b shows that the above constraint cannot be strengthened.

It should be mentioned that automorphism groups of Enriques surfaces remain a subject of intensive research. Much is known in the case of Enriques surfaces
with finite automorphism groups (even in positive characteristic) and unnodal Enriques surfaces, but a general picture is still missing. In this context
both the constraints given by Theorem~\ref{thm-orders} and the geometry of the families of Enriques surfaces discussed in Propositions~\ref{prop-f7gens},~\ref{prop:F9rho16},~\ref{prop:F9rho18}
are of separate interest. 
Still, such considerations exceed the scope of this paper. 

\noindent
We sketch our strategy for the proof of Theorem~\ref{main}.
The argument in \cite{hor} is based on criteria
for a polynomial to be the characteristic polynomial of
an isometry of a lattice. Unfortunately, all the six polynomials $F_{1},\dots ,F_{9},F_{15}$ do appear as factors
of modulo-$2$ reductions of characteristic polynomials
of isometries of the lattice $U\oplus E_8(-1)$ and the lattice $N$.
Thus we need to take Hodge structures and the ample cone into account as well. In this note we apply a modification of McMullen's method (see \cite{mcmullen11}, \cite{mcmullen16}) to obtain constraints on automorphisms of Enriques surfaces
that can realize the factors $F_7$, $F_9$, $F_{15}$. In particular, we can rule out the existence of the highest-degree factor $F_{15}$ (Prop.~\ref{prop}), and derive properties of the K3 covers of Enriques surfaces which realize
$F_7$ (Prop.~\ref{prop-one-in-genus}) and $F_9$ (Section~\ref{sec-F9}).
Then an algorithm based on Borcherd's method
 (\cite{borcherds1},  \cite{borcherds2}) and the ideas from
 \cite{shimada-algorithm}~and~\cite{BrandhorstShimada2019} allow us
 to find abstract Enriques surfaces realizing $F_7$ and $F_9$.
 For the readers convenience, the algorithm is presented in
 Section 6 in pseudocode.
 We close this section with a related open question.
 For an Enriques surface $\EnS$ we call the order of the image of bi-canonical representation
 \[\Aut(\EnS) \rightarrow \GL\left(H^0(\EnS,K_\EnS^{\otimes2})\right)\]
 of the automorphism group the transcendental index $I(\EnS)$ of $\EnS$.
 \begin{question}
  What are the
  possible transcendental indices of complex Enriques surfaces?
 \end{question}
  Note that all realizations of $F_7$ and $F_9$ must be by semi-symplectic automorphisms. Hence, we know that $7$ and $9$ do not divide $I(\EnS)$.
\par\medskip
\noindent
{\bf Notation:}
In this note,
we work over the field of complex numbers $\C$. 
Given a prime $p$, $\ZZ_p$ denotes the ring of $p$-adic integers.
For a ring $R$, we denote by $R^\times$ its group of units.
For a group $G$ and a prime $p$, $G_p$ is the $p$-Sylow subgroup of $G$.
\section{Preliminaries} \label{sect-bn}

\subsection*{Basic notation}
We maintain the notation of the previous section.
In particular, $\pi: \KtS \rightarrow \EnS$  is  the K3 cover of $\EnS$ and $\varepsilon$ is the covering involution of $\pi$.
Moreover, we have the finite index sublattice
\begin{equation} \label{eq-ds-varepsilon}
M \oplus N \subseteq H^2(\KtS,\Z)
\end{equation}
where  $M:= H^2(\KtS,\Z)^{\varepsilon}$ coincides with the pullback of $H^2(\EnS,\Z)$ 
by $\pi$
and $N := M^{\perp}$ (see e.g. \cite{namikawa}).
In particular, we have $M  \simeq U(2)\oplus E_8(-2)$ and  $N\simeq U\oplus U(2)\oplus E_8(-2)$,
where $U$ (resp.  $E_8$) denotes the unimodular hyperbolic plane (resp. the unique even unimodular
positive-definite lattice of rank 8). 
Let $f$ be an automorphism of $\EnS$.
The  sublattices $M$ and $N$
are preserved by the isometry $\tilde f^* \in \mbox{Aut}(H^2(\KtS,\Z))$, so
as in \cite{hor} we can
define the maps
$$
\tdfm :=  \tilde f^* \rvert_M \mbox{ and } \tdfn := \tilde f^* \rvert_N  \,
$$
and let $p_N$, $p_M$ (resp.  $\mu_N$, $\mu_M$) 
denote their characteristic (resp. minimal) polynomials. 
Then, (see e.g. \cite[\S 3]{hor}) we have
\begin{equation} \label{eq-div}
p_M \equiv p_f \bmod 2 \mbox{ and } (p_M  \bmod 2) \divides    (p_N \bmod 2).
\end{equation}
As we already mentioned,  
$\tdfn$ is a map of finite order (see e.g. \cite[Lemma~4.2]{third}), 
so  $p_N$  is a product of cyclotomic polynomials.

Recall that (see \cite[Prop~2.2]{ohashi-quotients}, \cite[Thm~1.1]{keum-quot})
\begin{equation} \label{eq-no-roots-in-N}
N \cap \NS(\KtS)  \mbox{ contains no vectors of square } (-2).
\end{equation}
Indeed, suppose to the contrary. 
By Riemann-Roch, a vector of square $(-2)$ in $N \cap \NS(\KtS)$ or its negative 
is the class of 
an effective divisor $C \in \NS(\KtS)$ such that
$\intf{\pi^*(D),C}=0$ for every $D \in \NS(\EnS)$.
This is impossible by the Nakai-Moishezon criterion, 
because we can choose $D$ so that
$\pi^*(D)$ is ample.

For an automorphism $f$ and an integer $k \in \N$ we define two lattices
\begin{equation} \label{eq-nk-def}
N_{k} := \mbox{ker}(\Phi_{k}(f_N)) \quad \mbox{ and } \quad  M_{k} := \mbox{ker}(\Phi_{k}(f_M)) \, .
\end{equation}
where  $\Phi_{k}(x)$ stands for the $k$-th cyclotomic polynomial.
Finally, to simplify our notation we put
$$
F_k(x) := (\Phi_{k}(x)  \bmod 2) \, .
$$
An automorphism $f$ of an Enriques surface is called \emph{semi-symplectic}, if it acts trivially on the global sections $H^0(\EnS,K_{\EnS}^{\otimes 2})$ of the bi-canonical bundle.
This is the case if and only both lifts $\tilde{f}$ and $\tilde{f} \circ \varepsilon$ of $f$ act on $H^0(X, \Omega^2_X)$ as $\pm 1$.
We denote by $\Aut_{s}(\EnS)$ the subgroup of semi-symplectic automorphisms.

\subsection*{Lattice}
Let $R \in \{\ZZ, \ZZ_p\}$ and $K$ be the fraction field of $R$. 
An \emph{$R$-lattice}
is a finitely generated free $R$-module equipped with a non-degenerate symmetric $K$-valued bilinear form $b$. 
If the form is $R$ valued, we call the lattice \emph{integral}. If further $b(x,x) \in 2 R$ for every $x \in L$, the lattice is called \emph{even}. The \emph{dual lattice} of $L$ is
\[L^\vee = \{ x  \in L \mid b(x,L) \subseteq R\}.\]
If $L$ is integral, then $L \subseteq L^\vee$ and we call the quotient $L^\vee/L$ the \emph{discriminant group} of $L$.
For $r \in R$, an $R$-lattice $L$ is called \emph{$r$-modular} if $rL^\vee = L$. If $r=1$, we call the lattice \emph{unimodular}.
The Gram matrix $G=(G_{ij})$ with respect to an $R$-basis $(e_1,\dots e_n)$ of $L$ is defined by  $G_{ij} = b(e_i,e_j)$.
The determinant $\det L \in R / R^{\times2}$ of $L$ is the determinant of any Gram matrix. For $R=\Z$ we have $|L^\vee / L| = |\det L|$.
The discriminant group carries the discriminant bilinear form induced by $b(x,y) \mod R$ for $x,y \in L^\vee$. If $L$ is an even lattice, its discriminant group moreover carries a torsion quadratic form induced by $x \mapsto b(x,x) \mod 2R$, called \emph{discriminant form}.
We say that two $R$-lattices $(L,b)$, $(L',b')$ are isomorphic if there is 
an $R$-linear isomorphism  $\phi: L \rightarrow L'$ such that
$b(x,x) = b'(\phi(x),\phi(x))$.
For $r \in R$ we denote by $L(r)$ the lattice with the same underlying free module as $L$ but with bilinear form $rb$.

Let $L$, $L'$, $L''$ be lattices.
The orthogonal direct sum of two lattices is denoted by $L \oplus L'$.
A sublattice $L'  \subseteq L$ is called \emph{primitive} if $L/L'$ is torsion free.  This is equivalent to $(L' \otimes K) \cap L = L'$. We call
\[L' \oplus L'' \subseteq L\] a {\sl primitive extension} if
$L'$, $L''$ are primitive sublattices of $L$ and $\rank L' + \rank L'' = \rank L$.
The finite group $L'' / (L \oplus L')$ is the \emph{glue} of the primitive extension. For any prime $p$ dividing its order, we say that $L$ and $L'$ are glued above/over $p$.
The signature (pair) $(s_+,s_-)$ of a $\ZZ$-lattice $L$ is the signature of $L\otimes \RR$ where $s_+$ is the number of positive and $s_-$ is the number of negative eigenvalues of a Gram matrix.
We denote by $U$ the even unimodular lattice of signature $(1,1)$.
By $A_n$ ($n\in \NN$), $D_n$ ($n\geq 4$), $E_6,E_7,E_8$
the positive definite root lattice with the respective Dynkin diagram.
\subsection*{Genus}
Two $\ZZ$-lattices $L$ and $L'$ are in the same \emph{genus} if $L \otimes \RR \cong L' \otimes \RR$ and for all prime numbers $p$ we have $L \otimes \ZZ_p \cong L' \otimes \ZZ_p$.
The genus is an effectively computable invariant and has a compact description in terms of
the so called \emph{genus symbols} introduced by Conway and Sloane
\cite[Chapter~15]{conway-sloane-book}. In what follows we give a short account.

Let $p$ be an odd prime. A $p$-adic unimodular lattice $L_1$ is determined up to isometry by its
rank $n_1$ and the $p$-adic square class $\epsilon_1 \in \{\pm 1\}$ of its determinant.
This is denoted by the symbol $1^{\epsilon_1 n_1}$.
Let $q=p^k$, and recall that a lattice is $q$-modular if it is of the form $L(q)$ for some unimodular $L$.
A $q$-modular $p$-adic lattice is given up to isomorphism by its scale $q$, its rank $n_q$ and the square class of the unit part of its determinant $\epsilon_q(L(q))
:= \epsilon_1(L)$. This is denoted by the symbol
$q^{\epsilon_q n_q}$.

A $p$-adic lattice $L$ admits a so called \emph{Jordan decomposition}
\[L = L_1 \oplus L_p \oplus \dots \oplus L_{p^k}\]
into $p^i$-modular lattices $L_{p^i}$. The latter are called the Jordan constituents. The decomposition is not unique. Nevertheless
we can compute the isomorphism class of $L$ from it. It is uniquely determined by the collection of
$(\epsilon_q,n_q)_q$ for the $q$-modular lattices $L_q$ as $q$ runs through the powers of $p$.
This collection is called the \emph{$p$-adic symbol} of $L$. We introduce the notation for $p$-adic symbols with an example.
The $3$-adic symbol
\[1^2 3^{-2} 27^5\]
denotes a $\ZZ_3$-lattice
\[L = L_1 \oplus L_3 \oplus L_{27}\]
such that $L_1$ is unimodular of rank $2$ and determinant a square,
$L_3$ is $3$ modular of rank $2$ and unit part of the determinant a non-square, that is
the determinant is $2 \cdot 3$ and the unit part is $2$ which is a $3$-adic non-square unit,
$L_{27}$ is $27$ modular of rank $5$ and the unit part of its determinant is a square.
We see that $L$ is isomorphic to $\ZZ_3^{9}$ with diagonal Gram matrix given by
\[\mbox{diag}(1,1,6,3,27,27,27,27,27).\]

An \emph{even} unimodular $2$-adic lattice $L_1$ is determined by its rank $n_1$ and $\epsilon_1 \in \{\pm 1\}$ which is
$1$ if the determinant is congruent to $1$ or $7$ modulo $8$ and $-1$ if it is congruent to $3$ or $5$.
This is denoted by $1^{\epsilon_1 n_1}$.
As before we obtain symbols for $q$-modular lattices and have a Jordan decomposition. A Jordan constituent is called \emph{even} if it is the twist of an even unimodular $2$-adic lattice.
A $2$-adic lattice all whose Jordan constituents are even is called \emph{completely even}.
Two completely even lattices are isomorphic if and only if
they have the same invariants $(\epsilon_q, n_q)$ for all powers $q$ of $2$.
If the lattices in question are not completely even, the classification involves an additional quantity called the oddity. However, in this note all lattices considered are 
completely even.

To describe a genus it is enough to give the signature pair and the local symbols at primes dividing twice the determinant. This is collected in a single symbol called the (Conway-Sloane) genus symbol.
For example $A_2$ is even of rank $2$ and has determinant $3$. In particular it is $2$-adically unimodular and has the $2$-adic symbol $1^{-2}$. To compute the $3$-adic symbol, we note that it is $3$-adically equivalent to the lattice diag$(2,6)$  with $3$-adic symbol $1^{-1}3^{-1}$.
Together this gives
\[ \even_{(2,0)}1^{-1}3^{-1}.\]
Here the $\even$ indicates that this lattice is even and the index $(2,0)$ that it is positive definite of rank $2$.
Finally, the unimodular Jordan constituents can be reconstructed from the determinant.
Thus they are omitted and the symbol is abbreviated to $\even_{(2,0)}3^{-1}$.

Note that Conway and Sloane give necessary and sufficient conditions on when a collection of local symbols defines a
non-empty genus \cite[Thm~15.11 on p.~383]{conway-sloane-book}.

\begin{remark} \label{remark-computer-genus}
The genus symbols and their relation with discriminant forms are implemented in sageMath \cite{sage}
by the first author. It is possible to compute all classes in a genus using Kneser's neighboring algorithm \cite{neighbor} and Siegel's mass formula.
Similarly roots can be found using short vector enumerators \cite[$\S$.2.7.3]{cohen}. We used the implementation provided by PARI \cite{pari} via sageMath.
\end{remark}

In the following we relate the genus symbols with primitive extensions and isometries.
\begin{lemma}\label{lem-directsum}
 Let $L$ and $L'$ be completely even $p$-adic lattices with symbols $(\epsilon_q,n_q)_q$ respectively $(\epsilon'_q,n'_q)_q$ then $L \oplus L'$ has symbol
 $(\epsilon_q \epsilon'_q, n_q + n_q')$.
\end{lemma}
\begin{proof}
 If $\bigoplus L_q$ and $\bigoplus L'_q$ are the respective Jordan decompositions, then
 $\bigoplus (L_q \oplus L_q')$ is a Jordan decomposition of the sum. Finally the square class
 is multiplicative and the rank is additive.
\end{proof}

\begin{lemma}\label{lem-gluesymbol}
 Let $L$ and $L'$ be completely even $p$-adic lattices with symbols $(\epsilon_q,n_q)_q$ and $(\epsilon'_q,n'_q)_q$.
 Then there is a primitive extension $L \oplus L' \subseteq L''$ with $L''$ unimodular if and only if for all $q > 1$
  $n'_q = n_q$ and $\epsilon'_q = \delta ^ {n_q} \epsilon_q$
 where $\delta = \begin{cases}
    \phantom{-}1 \quad\mbox{ for } p \equiv 1,2 \mod 4\\
             -1 \quad \mbox{ for } p \equiv 3\phantom{,2} \mod 4.
                                 \end{cases}$

 \end{lemma}
\begin{proof}
 From \cite[Cor. 1.6.2]{nikulin-sym} we know that the existence of a unimodular primitive extension is equivalent to existence of an anti isometry
 of the discriminant forms of $L$ and $L'$. Since the lattices are completely even, this means precisely that the Jordan constituents of scale $q>1$ are anti isomorphic.
 If $L$ is a $q$-modular lattice with symbol $(n_q,\epsilon_q)$, then $L(-1)$ has determinant $(-1)^{n_q}\det L$. Hence the symbol of $L(-1)$ is $(n_q,\delta^{n_q}\epsilon_q)_q$ where for $p\neq 2$, $\delta$ is $1$ or $-1$ according to $-1$ being a $p$-adic square or not. If $p=2$, then $\delta = 1$.
\end{proof}

In the sequel we will apply the following lemma.
\begin{lemma}\label{lemma:2adic_symbol_N3}
 Let $L$ be a $\ZZ$-lattice and let $g \in \OG(L)$ be an isometry with minimal polynomial $\Phi_3$. Then $L$ is completely even and the $2$-adic symbols of the genus of $L$ are of the form
$$
q_i^{\epsilon_i n_i} \quad \mbox{ where } q_i = 2^i, \, n_i \mbox{ is even and } \epsilon_i = (-1)^{n_i/2}.
$$

\end{lemma}
\begin{proof}
This is a special case of \cite[Prop. 2.17, Kor. 2.36]{hoeppner:dissertaion}.
\end{proof}
In particular, when $L$ is a rank-$2$ (resp. rank-$4$) lattice  of discriminant at most $4$ (resp.  $16$) its $2$-adic symbols are
$1^{-2}$, $2^{-2}$ (resp. $1^4$, $1^{-2}2^{-2}$, $2^4$, $1^{-2}4^{-2}$)

\subsection*{\texorpdfstring{$\Phi_n(x)$}{P(x)}-lattices} In the sequel we need the notion of a $\Phi_n(x)$-lattice.
  The reader can consult
\cite{Gross--McMullen}, \cite[$\S$~5]{mcmullen16} for a concise and more general exposition of the facts we briefly sketch below. \\
Recall that a $\Phi_n(x)$-lattice is defined to be a pair $(L, f )$ where $L$ is an integral lattice and $f \in \OG(L)$ is
an isometry with characteristic
polynomial $\Phi_n(x)$.\\
Let $n > 2$, the principal $\Phi_n(x)$-lattice $(L_0, \langle \cdot,\cdot \rangle_0,f_0)$  is defined as the $\ZZ$-module $L_0 := \ZZ[\zeta_n]$ equipped with the scalar product
\[\langle g_1, g_2 \rangle_0 = \tr^{\QQ[\zeta_n]}_\QQ \left(\frac{g_1\overline{g_2}}{r'_n(\zeta_n+\zeta_n)} \right) \]
where $\tr$ is the field trace of $\QQ[\zeta_n]/\QQ$, $r_n\in \QQ[x]$ is the minimal polynomial of $\zeta_n + \zeta_n^{-1}$, and $r_n'$ is its derivative.
Finally, $f_0\colon L_0 \rightarrow L_0$, $x\mapsto \zeta_n \cdot x$, is an isometry with minimal polynomial $\Phi_n$.
One can show that $L_0$ is an even lattice and
\begin{equation} \label{eq-princ-det}
\det(L_0) = |\Phi_n(1) \Phi_n(-1)|.
\end{equation}
Given a pair $(L,f)$ as above and an element $a \in \ZZ[f + f^{-1}] \subset \End(L)$ one can define another inner product on $L$
by the formula $\langle g_1, g_2 \rangle_a := \langle ag_1, g_2 \rangle_0$. We say that the resulting lattice is the twist of $L$ by $a$ and denote it by $L(a)$.
Recall, that by \cite[Thm~5.2]{mcmullen16}
\begin{equation}\label{eq-principal}
\mbox{every } \Phi_n(x)\mbox{-lattice is a twist of the principal lattice } (L_0, \langle \cdot,\cdot \rangle_0,f_0).
\end{equation}
The genus symbols of $\Phi_n$-lattices are computed in \cite[Satz 2.57]{hoeppner:dissertaion}.
\subsection*{Equivariant gluing}

We note the following well known Lemma for later use.
\begin{lemma}\label{lem-det-index-formula}
 If $A \oplus B \subseteq C$ is a primitive extension, then
 \[\det A \det B =   [C:A\oplus B]^2 \cdot \det C\]
and
 \[\det A \mid  [C:A\oplus B] \cdot \det C.\]
 Moreover, if $p$ is a prime such that
 $p \nmid [C:A \oplus B]$,
 then \[C\otimes \ZZ_p = (A\otimes \ZZ_p) \oplus (B \otimes \ZZ_p).\]
 \end{lemma}

Let $a \in \OG(A), b\in \OG(B), c\in \OG(C)$ be isometries.
We call $(A,a) \oplus (B,b) \subseteq (C,c)$ an equivariant primitive extension if the restriction $c|_{A\oplus B} = a\oplus b$.

\begin{lemma}\label{lem-glue-resultant}
 Let $(A,a) \oplus (B,b) \hookrightarrow (C,c)$ be an equivariant
 primitive extension with characteristic polynomials $p_A,p_B$.
 Then any prime dividing the index $[C:A \oplus B]$ divides the resultant $\res(p_A,p_B)$.
\end{lemma}
\begin{proof}
 Apply \cite[Prop. 4.2]{mcmullen16} to $G = C/(A \oplus B)$.
\end{proof}

\begin{lemma} \label{lem-useful}
 Let $(A,a) \oplus (B,b) \hookrightarrow (C,c)$ be an equivariant
primitive extension.
 Suppose that the characteristic polynomial $p_a$ of $a$ is $\Phi_n(x)$.
 Then the glue $G = C/(A \oplus B)$ satisfies
 \[ |G| \mid \operatorname{res}(\Phi_n,\mu)\]
 where $\mu=\mu_b$ is the minimal polynomial of $b$.
\end{lemma}
\begin{proof}
Let $G_A$ denote the orthogonal projection of $G$ to $A^\vee/A$
and $\overline{a}$ the automorphism on $G_A$ induced by $a$.
Since $G_A$ is a finite $\Z[\zeta_n]$-module generated
by one element, we have
$G_A = \Z[\zeta_n]/I$ where $I$ is the  kernel of the map $\Z[\zeta_n] \mapsto \End G_A$ that sends the root of unity $\zeta_n$ to
$\overline{a}$. This yields:
\[
\mu(\overline{a}) = 0 \mbox{ thus } \mu(\zeta_n) \in I
\]
and
\[
|G| = |G_A|=|{\mathcal O}_K/I| = N(I) \mid N(\mu(\zeta_n)) = \prod_{(k,n)=1}
\mu(\zeta_n^k) = \mbox{res}(\phi_n,\mu_b)
\]
where $N(I)$ is the norm of the ideal $I$.
\end{proof}


\section{Ruling out the factor \texorpdfstring{$F_{15}$}{F15}} \label{sect-th1}

The main aim of this section is to prove the following proposition.

\begin{proposition} \label{prop}
Let $f$ be an automorphism of an Enriques surface $\EnS$ and let $p_f$ be the minimal polynomial of the map $f^*: \NumS \rightarrow \NumS$.
Then the modulo-$2$ reduction $(p_f(x)  \bmod 2)$ is never divisible by the polynomial
$$
F_{15} =    x^8 + x^7 +       x^5 + x^4 + x^3 +       x + 1
$$
i.e. by the modulo-$2$ reduction  of the cyclotomic polynomial $\Phi_{15}(x)\in \mathbb{Z}[x]$.
\end{proposition}

Recall (see e.g. \cite{cantat})  that $p_f$ is a product of cyclotomic polynomials and at most one Salem factor.
Since $p_f$ is reciprocal, $(p_f(x)  \bmod 2)$ is divisible by an irreducible factor of $F_{15}$ if and only if it is divisible by the whole $F_{15}$ (c.f. \cite{hor}).   \\

\noindent
{\bf Proof of Prop.~\ref{prop}}
Assume that $F_{15} \divides    (p_f \bmod 2)$. Combined with \cite[Remark~2.4]{hor}, 
this implies that
\begin{equation}\label{eq-F15}
(p_M \bmod 2) = F_{15} \cdot  F_{1}^2   \quad \mbox{ and }  \quad (F_{15} \cdot F_{1}^2)  \divides    (p_N \bmod 2) .
\end{equation}

By \cite[Lemma 2.1]{hor} and \cite[Lemma 2.5]{hor} the charateristic polynomial $p_N$ is  a product of cyclotomic polynomials of degree at most $8$. Computing modulo-$2$ reductions of all
such cyclotomic polynomials, one infers that
either $\Phi_{15} \divides p_N$ or  $\Phi_{30} \divides p_N$.
Replacing $\tilde{f}$ by a power coprime to  $15$ we
can assume that
$p_N$ is a product of the $\Phi_k$ for $k \in \{1,3,5,15\}$.
Together with (\ref{eq-F15}) this leaves us with the two possibilities
\begin{equation}
 p_N = \Phi_{15} \cdot \Phi_{1}^4 \qquad \mbox{or} \qquad p_N = \Phi_{15} \cdot \Phi_{3} \cdot\Phi_{1}^2.
\end{equation}

We consider the (primitive) $f_N$-invariant sublattice
$N_{15}$ (see \eqref{eq-nk-def}) 
and denote its orthogonal complement in $N$ by $N_{15}^{\perp}$.
Since $\Phi_{15}(x)$ has no real roots, the signature of $N_{15}$ is of the form $(2k,2(4-k))$ with $k \in\{1,2,3,4\}$.
Recall that $N$ is of signature $(2,10)$ and contains
$N_{15}$. Thus the signature of $N_{15}$ is either
$(0,8)$ or $(2,6)$.

By definition
$$
N_{15} \oplus N_{15}^{\perp} \subset N
$$
is a primitive extension. Let $G=N/(N_{15} \oplus N_{15}^{\perp})$ be the glue between $N_{15}$ and $N_{15}^{\perp}$.
Then by Lemma \ref{lem-useful} we have

\[
|G| \divides   \mbox{res}(\Phi_{15}, \mu_{f|_{N_{15}^{\perp}}})
\]
But we have
\begin{equation} \label{eq-resultantcyclo}
\mbox{res}(\Phi_{15}, \Phi_1) = 1 \mbox{ and } \mbox{res}(\Phi_{15}, \Phi_3) = 25 \, .
\end{equation}
In particular, if $|G| > 1$ then
\begin{equation} \label{eq-phi-manyfact3}
p_N = \Phi_{15} \cdot \Phi_{3}  \cdot \Phi_{1} ^2.
\end{equation}
In what follows we treat the cases whether $G$ is trivial or not separately.\\
\subsection*{The case when \texorpdfstring{$\mathbf{G}$}{G} is trivial.}
Assume that the glue $G$ is trivial, i.e.
\begin{equation} \label{eq-N-direct-sum}
N_{15} \oplus N_{15}^{\perp} = N \in \even_{(2,10)}2^{10}.
\end{equation}
Let $(\epsilon_q,n_q)$ be the $2$-adic genus symbol of $N_{15}$ and
$(\epsilon'_q,n'_q)$ the symbol of $N_{15}^\perp$.
From Lemma \ref{lem-directsum} we infer that $10 = n_2 + n'_2$.
Further $n'_2 \leq \rank N_{15}^\perp = 4$ and $n_2 \leq \rank N_{15} = 8$. Thus we obtain $6 \leq n_2 \leq 8$.
Since $N_{15}$ is a $\Phi_{15}$-lattice, we can calculate all $\Phi_{15}$-lattices matching this condition. There is exactly one such lattice up to isometry:

\begin{equation} \label{eq-oguiso-1}
 N_{15} \cong E_8(-2) \in \mbox{II}_{(0,8)}2^{8} .
\end{equation}
Using Lemma \ref{lem-directsum} once more, we calculate the genus symbol of $N_{15}^\perp$ from those of $N$ and $N_{15}$ and see that

\begin{equation} \label{eq-oguiso-2}
N_{15}^{\perp} \cong U \oplus U(2) \in \mbox{II}_{(2,2)}2^{2}
\end{equation} is the unique class in its genus.
From \eqref{eq-oguiso-1}, \eqref{eq-oguiso-2} and \cite[Lemma~7.7]{oguiso-yu} we infer that the spectral radius of $\tdfm$ is one (i.e. $f$ has trivial entropy).
Thus $p_M$ is not divisible by a Salem polynomial and
must be a product of cyclotomic polynomials.
A direct computation of modulo-$2$ reductions of all cyclotomic polynomials of degree at most $8$
shows that either $\Phi_{30}$ or $\Phi_{15}$ divides  $p_M$. By replacing $\tilde{f}$ with its iteration (i.e. by $\tilde{f}^2$ or $\tilde{f}^4$) we can assume that
$$
p_M = \Phi_{15} \cdot \Phi_{1}^2 \, .
$$

We consider the rank $2$ lattice $M_{1}$ and the rank $8$ lattice $M_{15}$ (see \eqref{eq-nk-def}).
Because  $\Phi_{15}$ has no real roots, $M_{15}$ has signature $(2k, 2(4-k))$. But
$M$ is of signature $(1,9)$, so  the lattice $M_{15}$
is negative-definite.
Since the resultant  $\res(\Phi_{15},\Phi_{1})$ is trival, there is no glue between $M_1$ and $M_{15}$ which leaves us with $M_{15} \cong E_8(-2)$ and $M_1\cong U(2)$.
We observe that
\[H_{15}:=\ker \Phi_{15}(\tilde{f}^*) = \overline{M_{15} \oplus N_{15}}\]
is the primitive closure of $M_{15} \oplus N_{15}$ in  $H^2(\KtS,\ZZ)$.
Since the resultant $\mbox{res}(\Phi_{15}, \Phi_{1}\Phi_{3})=25$ is odd,
there is no glue over $2$ between $H_{15}$ and $H_{15}^\perp$. As further $\disc (M_{15} \oplus N_{15})$ and hence $\disc(H_{15})$ is not divisible by $3$, there is no glue above $3$ either. Thus $H_{15}$ is a direct summand of the unimodular lattice $H^2(\KtS,\ZZ)$. In particular, it
is an  even negative-definite, unimodular  lattice of rank $16$.
Such a lattice is either the direct sum of two copies of $E_8(-1)$ or it is the
even negative-definite, unimodular lattice  $\Gamma_{16}$ whose root sublattice is $D_{16}(-1)$ (see e.g. \cite[Table~1]{conway-sloane-88}).
Each of those lattices has roots, so we can find a root in $\NS(\KtS) \supseteq H_{15}$. By Riemann-Roch such a root defines an effective divisor $C \in \NS(\KtS)$ such that
$$
C + \tilde{f}^*C + \ldots + (\tilde{f^*})^{14}C = 0 \in H^2(\KtS,\ZZ)
$$
and we arrive at  a contradiction because $\NS(\KtS)$ contains an ample class 
(c.f.  \cite[\S 2]{mcmullen16}). Hence the glue $G$ cannot be trivial.
\subsection*{The case when G is non trivial.}
Assume that the glue $G$ is non-trivial. In particular,  the characteristic polynomial $p_N$ satisfies  \eqref{eq-phi-manyfact3}. \\
By Lemma~\ref{lem-useful} and \eqref{eq-resultantcyclo} we have either $|G|=5$ or $|G|=25$. Thus Lemma \ref{lem-det-index-formula} with \cite[Prop. 5.1]{BF} implies that $|G|=25$ and
$$
\disc(N_{15}) \cdot \disc (N_{15}^{\perp}) = \disc (N)  \cdot |G|^2 = 2^{10} \cdot 5^4
$$
Observe that $\disc(N_{15})=2^{8k}r$ for some $k,r \in \N$, so we have
\begin{equation} \label{eq-discrn15}
\disc(N_{15}) = 2^{8} \cdot 5^2.
\end{equation}
One computes that the genus of a $\Phi_{15}$-lattice with this determinant and signature either $(2,6)$ or $(0,8)$ is unique. It is given by
\[N_{15} \in \even_{(2,6)}2^8 5^{-2} \mbox{ and thus } N_{15}^\perp \in \even_{(0,4)}2^2 5^{-2}\]
(using Lemmas \ref{lem-directsum} and \ref{lem-gluesymbol}).
Since $\res(\Phi_3,\Phi_1)=3$ is odd, we know that
\[N_{15}^\perp\otimes \ZZ_2 = (N_1 \otimes \ZZ_2) \oplus (N_3 \otimes \ZZ_2).\]
The rank of $N_3$ is $2$, so by Lemma \ref{lemma:2adic_symbol_N3} the $2$-adic symbol of $N_3$ is $q^{-2}$ for $q=2^i$.
The $2$-adic symbol of $N_{15}^\perp$ is $1^2 2^2$.
By Lemma \ref{lem-directsum}, $i \leq 1$, and if $i=0,1$, then the sign is wrong. Hence $N_3\otimes \ZZ_2$ cannot be a direct summand of $N_{15}^\perp\otimes \ZZ_2$ which is a contradiction.
\hfill $\Box$

\section{The factor \texorpdfstring{$F_9$}{F9}} \label{sec-F9}

In this section we maintain the notation of  previous sections and prove Theorems~\ref{main},~\ref{thm-orders}. 
We assume that $f \in \Aut(\EnS)$ satisfies the condition
\begin{equation} \label{eq-f9-div}
F_{9} \divides    (p_f \bmod 2) \, .
\end{equation}
After replacing $\tilde{f}$ by some power co-prime to $3$ we may assume that $f_N$ is of order $9$.
Since $F_9F_1^2$ divides $p_N$, we can rule out $p_N = \Phi_9^2$. This leaves us with the three possibilities
\begin{equation} \label{eq-f9-possibilities}
p_N = \Phi_9 \Phi_3^{k}\Phi_1^{6-2k} \qquad k \in \{0,1,2\}.
\end{equation}
As usual we set $N_9 := \ker(\Phi_9(f_N))$ and denote by $N_9^\perp$ the orthogonal complement of $N_9$ in $N \in \even_{(2,10)}2^{10}$.
By Lemma \ref{lem-useful} $\det N_9 \mid 2^6\res(\Phi_9,\Phi_3\Phi_1)=2^6 \cdot3^3$.
Using the description of $N_9$ as $\Phi_9$-lattice, we enumerate the possibilities for $N_9$. This yields $4$ cases and with Lemmas \ref{lem-directsum} and \ref{lem-gluesymbol} we calculate the corresponding genus of $N_9^\perp$.
\begin{equation}\label{N9-1}
N_9 \in \even_{(0,6)}2^{-6}3^{1} \mbox{ and } N_9^\perp \in \even_{(2,4)}2^{-4}3^{-1}
\end{equation}
\begin{equation}\label{N9-2}
N_9 \in \even_{(0,6)}2^{-6}3^{-3} \mbox{ and } N_9^\perp \in \even_{(2,4)}2^{-4}3^{3}
\end{equation}
\begin{equation}\label{N9-3}
N_9 \in \even_{(2,4)}2^{-6}3^{-1} \mbox{ and } N_9^\perp \in \even_{(0,6)}2^{-4}3^{1}
\end{equation}
\begin{equation}\label{N9-4}
N_9 \in \even_{(2,4)}2^{-6}3^{3} \mbox{ and } N_9^\perp \in \even_{(0,6)}2^{-4}3^{-3}
\end{equation}
We can rule out the cases (\ref{N9-3}) and (\ref{N9-4}) since in each case the genus of $N_9^\perp$ consists of a single class
(see Remark~\ref{remark-computer-genus}) which contains roots.
We continue by determining the characteristic polynomial.

\begin{lemma}
 Let $g \in \OG(N)$ be an isometry of order $9$,
 then the characteristic polynomial $p_N$ of $g$ is not of the form
\[p_N = \Phi_9 \Phi_3^{2}\Phi_1^{2}.\]
\end{lemma}
\begin{proof}
Suppose that $p_N = \Phi_9 \Phi_3^2 \Phi_1^2$.
Recall that
by Lemmas \ref{lem-det-index-formula} and \ref{lem-glue-resultant}.
\[N \otimes \ZZ_2 = (N_9 \otimes \ZZ_2) \oplus (N_3\otimes \ZZ_2) \oplus (N_1\otimes \ZZ_2).\]
We see that $N_3\otimes \ZZ_2$ is of rank $4$ and has maximal scale of a $2$-adic Jordan component equal to $2$.
By Lemma \ref{lemma:2adic_symbol_N3} the possible $2$-adic symbols
of $N_3$ are $1^{4}$, $1^{-2} 2^{-2}$ and $2^{4}$. \\
In all cases \eqref{N9-1} - \eqref{N9-4} the $2$-adic symbol of $N_9^\perp$  is $1^2 2^{-4}$.
Therefore $N_3\otimes \ZZ_2$ cannot be a direct summand of $N_9^\perp\otimes \ZZ_2$. Indeed, in the first case $1^{4}$ the unimodular part is too big.
In the second case $1^{-2} 2^{-2}$ the unimodular part has the wrong determinant, and finally in the last case $2^{4}$ the $2$-modular part has wrong determinant. This contradiction completes the proof.
\end{proof}
If $p_N=\Phi_9 \Phi_1^6$, then we must be in case \eqref{N9-1} and $N_9^\perp = N_1$. Since the signature of $N_1$ is $(2,4)$, it contains the transcendental lattice. In particular, $f$ is semi-symplectic. Choosing the covering K3 surface general enough, we may assume that $N_1$ is its transcendental lattice. This situation is analyzed in the next
\begin{proposition}\label{prop:F9rho16}
 Let $\EnS$ be an Enriques surface such that its covering K3 surface $\KtS$ has transcendental lattice
 \[T(X) \cong U \oplus U(2) \oplus A_2(-2) \in\even_{(2,4)}2^{-4}3^{-1}\]
and satisfies the condition
 \[N \cap \NS(\KtS) \cong E_6(-2) \in \even_{(0,6)}2^{-6}3^{1}.\] 
Then the image of $\Aut_{s}(\EnS)\rightarrow \OG(\NumS) \otimes \FF_2$ generates a group isomorphic to $\mathcal{S}_5$.
\end{proposition}
\begin{proof}
The image of $\Aut_{s}(\EnS)\rightarrow \OG(\NumS)$ can be calculated with Algorithm~\ref{algo:main}.
 It is generated by $64$ explicit matrices (see \cite{shimada-comp-BRS}). 
 Their mod $2$ reductions generate a group isomorphic to $\mathcal{S}_5$. The latter can be checked with help of \cite{gap}.
\end{proof}

Since $\mathcal{S}_5$ does not contain an element of order $9$, we are left with
\[p_N= \Phi_9 \Phi_3 \Phi_1^4.\]
We derive further restrictions.
\begin{lemma} \label{lem:N3-A2}
 Let $g \in \OG(N)$ be an isometry with characteristic polynomial $$p_N= \Phi_9 \Phi_3^{1}\Phi_1^{4}.$$
 Then $N_3 = A_2(n)$ with $n \in \{\pm 2, \pm 6\}$.
\end{lemma}
\begin{proof}
One can easily see that $A_2$ is the principal $\Phi_3$-lattice.
By \eqref{eq-principal} $N_3=A_2(n)$ for some $n \in \ZZ$.
In the following we show that  $n \in \{\pm 2, \pm 6\}$ by bounding the determinant of $N_3$.
By Lemma~\ref{lem-useful} we have
\[\det N_3 \mid 2^2\res(\Phi_3,\Phi_9 \Phi_1)=2^2 3^3.\]
By Lemma \ref{lemma:2adic_symbol_N3} the $2$-adic symbol of $N_3$ is either $1^{-2}$ or $2^{-2}$. The first one is not a direct summand of $N_9^\perp \otimes \ZZ_2$ (see Lemma \ref{lem-directsum}), so we are left with the second. Hence $|n| \neq 1$.
\end{proof}

\begin{lemma}
 Let $f \in \Aut(\EnS)$ be an automorphism of an Enriques surface such that
 $p_N= \Phi_9 \Phi_3^{1}\Phi_1^{4}$ and (\ref{N9-1}) holds.
 Then $N_3 \cong A_2(-2)$ and $N_1\cong U(2) \oplus U$.
\end{lemma}
\begin{proof}
 By assumption (\ref{N9-1}) $\det N_9^\perp=2^4 3$, and Lemma \ref{lem-useful} yields $\det N_3 \mid 2^2 9$. Thus by Lemma \ref{lem:N3-A2}, we are left with $N_3 = A_2(\pm 2)$.
 We see that $\det N_1 \mid 2^2 3^2$.
 Suppose that $N_3 = A_2(2) \in \even_{(2,0)}2^{-2}3^{1}$.
 There is a single genus of signature $(0,4)$, $2$-adic symbol $1^{2} 2^2$ and determinant dividing $2^2 3^2$, namely $N_1 \in \even_{(0,4)}2^2 3^2$. It consists of a single class which has roots.
 Thus $N_3 \cong A_2(-2)$. We calculate the possible genus symbols of $N_1$ as $\even_{(2,2)}2^2$ and $\even_{(2,2)}2^2 9^{\pm 1}$.
 In the second case $N_1$ and $N_3$ must be glued non-trivially over $3$. This is impossible, as the only possibility for the glue groups are $(N_3^\vee/N_3)_3$ whose discriminant form is non-degenerate and
 $3(N_1^\vee/N_1)_3$ whose discriminant form is degenerate.
 Thus $N_1 \in \even_{(2,2)}2^2$ which implies $N_1 \cong U(2) \oplus U$ since it is unique in this genus.
\end{proof}

If the transcendental lattice is $U \oplus U(2)$, then as before we see that the spectral radius of $\tilde{f}$ is one.
Since $M_1$ is of rank $2$ and $f_M|M_1$ has spectral radius zero, it is of finite order. Since $M_1^\perp$ is definite $f_M$ is of finite order there as well. Thus $\tilde{f}$ is an automorphism of order $9$ on a complex Enriques surface. However no such isomorphism exists (cf. \cite{MO1}).
We are left with case (\ref{N9-2}) and $p_N=\Phi_9 \Phi_3 \Phi_1^4$.
\begin{lemma}
 Let $f \in\Aut(\EnS)$ be an automorphism of an Enriques surface such that
 $p_N= \Phi_9 \Phi_3^{1}\Phi_1^{4}$ and (\ref{N9-2}) holds.
 Then $N_3 \cong A_2(-6)$ and $N_1 \in \even_{(2,2)}2^{-2}9^{1}$. Moreover $N_1^\perp \cong A_8(-2)$.
\end{lemma}
\begin{proof}
Recall that $\zeta_9\cdot x:=g(x)$ defines a $\ZZ[\zeta_9]$-module structure on $N_9$ and its discriminant group. Thus $N_9^\vee / N_9 \cong \ZZ[\zeta_9]/I$ for some ideal $I$.
Since we are in case (\ref{N9-2}), $I$ is of norm $\det N_9 = 2^63^3$. There is only one such ideal, namely $2(1-\zeta_9)^3$. We see that the action of $g$ on the $3$-primary part $(N_9^\vee/N_9)_3\cong \ZZ[\zeta_9]/(1-\zeta_9)^3$ has minimal polynomial $(x-1)^3=x^3-1$. In particular it has order $3$. Thus the order of $g$ on \[\left(N_9^{\perp\vee} \! / N_9^\perp\right)_3 \cong (N_9^\vee/N_9)_3\] is $3$ as well. This is only possible if the order of $g$ on $\left(N_3^\vee / N_3\right)_3\cong \ZZ[\zeta_3]/(1-\zeta_3)^i$ is $3$ (this group is a subquotient of $(N_3\oplus N_1)^\vee / (N_3\oplus N_1)$ ). This implies that $i\geq 2$, i.e. that
$\det N_3$ is divisible by $9$. From Lemma \ref{lem:N3-A2} we see that $N_3 = A_2(\pm 6)$.
Now that we know the determinant of $N_3$ and $N_9^\perp$, we can estimate that of $N_1$ to be a divisor of $2^2 3^2$.
Since $N_3$ has a $3$-adic Jordan component of scale $9$ and $N_9^\perp$ not, $N_3$ cannot be a direct summand of $N_9^\perp$. Thus $N_3$ and $N_1$ are glued non-trivially over $3$. Consequently the determinant of $N_1$ is $2^2 3^2$.

Suppose that $N_3\cong A_2(6)$, then the signature of $N_1$ is $(0,4)$.
There is only one genus with $2$-adic genus symbol $1^2 2^2$, signature $(0,4)$ and determinant $2^2 3^2$: $\even_{(0,4)} 2^2 3^2$ it consists of a single class which has roots.

Suppose now that $N_3 \cong A_2(-6)$. Then we obtain $3$ possibilities for the genus of $N_1$:
\begin{enumerate}
 \item $\even_{(2,2)} 2^2 3^{-2}$; There is only one possibility to glue $N_3$ and $N_1$ equivariantly over $3$ (up to isomorphism). It results in $\even_{(2,4)}2^{-4} 3^1 9^1 $ which is not what we need;
 \item $\even_{(2,2)} 2^2 9^{-1}$; the full $3$-adic symbol is $1^{-3} 9^{-1}$. But that has the wrong sign at scale $1$.
 \item $\even_{(2,2)} 2^2 9^{1}$ indeed there is a unique possibility to glue $N_3$ and $N_1$ equivariantly over $3$. It yields the correct result.
\end{enumerate}
\end{proof}
\begin{corollary}
 If $F_9$ divides $(p_f \bmod 2)$, then $F_1^2F_3F_9$ divides $(p_f \bmod 2)$.
\end{corollary}
\begin{proof}
 By the previous proposition $(N_3^\vee/N_3)_2 \cong \FF_2^2$. Hence $F_3$ divides $p_N\mod 2$.
 Since $p_f \! \! \mod 2 \mid p_N \!\! \mod 2= F_9 F_3 F_1^4$, the corollary is proven.
\end{proof}

We have determined the N\'eron-Severi lattice of the K3 cover of a
generic Enriques surface admitting an automorphism with $F_9$ dividing $p_f \mod 2$. This allows us to compute the semi-symplectic part of the automorphism group and locate $f$ in there.
\begin{proposition}\label{prop:F9rho18}
 Let $\EnS$ be an Enriques surface such that its K3 cover $\KtS$
satisfies the condition
\[\NS(\KtS) \cap N \cong A_8(-2) \in \even_{(0,8)}2^8 9^1\]
and  has
the  transcendental lattice given by
  \[N_1 \in \even_{(2,2)}2^{-2}9^{1} .\]
Then, the image of $\Aut_{s}(\EnS) \rightarrow \OG(\NumS \otimes \FF_2)$ generates a group isomorphic to $\mathcal{S}_9$. \\
In particular, the polynomials $F_7$ and $F_9$ do appear as factors of modulo-$2$ reductions of characteristic polynomials  
of isometries induced by some automorphisms of the Enriques surface $\EnS$.
  \end{proposition}
\begin{proof} The proof is a direct computation with the help of Algorithm~\ref{algo:main} (c.f. proof of Prop.~\ref{prop:F9rho16}).
The existence of the factors $F_7$ and $F_9$ follows since the symmetric group $\mathcal{S}_9$ has elements of order $7$ and $9$.
\end{proof}

Finally we can give the proofs of the main results of this note.

\vspace*{1ex}
\noindent
{\bf Proof of Theorem~\ref{main}}  a) One can repeat verbatim the proof of 
\cite[Theorem 1.2]{hor} to see that the modulo-$2$ 
reduction $(p_N(x)  \bmod 2)$ is the product of some of the polynomials
$F_{1}, F_{3}, F_5, F_7, F_9, F_{15}$. By \eqref{eq-div} the same holds for $(p_f(x)  \bmod 2)$. The claim follows from
Prop.~\ref{prop}. \\
b) follows from Prop.~\ref{prop:F9rho18}.
\hfill $\Box$

\vspace*{1ex}
\noindent
{\bf Proof of Theorem~\ref{thm-orders}}
 If the order of $f_N$ is $90,45,72$, then $F_9$ divides $p_N \mod 2$. Hence, by the previous corollary, $p_N \mod 2$ is divisible by $F_1^2F_3F_9$.
 In particular, $p_N$ (of degree 12) cannot be divided by $\Phi_5$ as well. This excludes orders $45$ and $90$.
 If the order is $72$, then the characteristic polynomial must be divisible by $\Phi_8$ and by one of $\Phi_{3a} \Phi_{9b}$ with $a,b \in \{1,2\}$.
 From the previous considerations we know that $N_8 \in \even_{(2,2)}2^{-2}9^{1}$. This is impossible, as can be seen using the description of $N_8$ as a twist of the principal $\Phi_8$-lattice. \hfill $\Box$


\section{The factor \texorpdfstring{$F_7$}{F7}}

The main aim of this section is to study Enriques surfaces $\EnS$ with an automorphism  $f \in \Aut(\EnS)$ such that
\begin{equation} \label{eq-f7div}
F_{7} \divides    (p_f \bmod 2) \, .
\end{equation}
The existence of such surfaces follows from Prop.~\ref{prop:F9rho18}.
Here we derive a lattice-theoretic constraint given by  \eqref{eq-f7div} and show that it indeed defines 
Enriques surfaces with the desired property.
We maintain the notation of the previous sections. Recall (see \eqref{eq-ds-varepsilon}) that
$$
N \in \even_{(2,10)}2^{10} .
$$
In the sequel we will need the following lemma.
\begin{lemma} \label{lemm-twocases}
Let $g\in \OG(N)$ be an isometry such that
its characteristic polynomial is the product  $\Phi_7(x)\Phi_1(x)^6.$
Then there are two possibilities for the genera of the lattices $N_7 := \ker \Phi_7(g)$  and $N_1 := \ker \Phi_1(g)$; either
\[ N_7 \in \even_{(2,4)}2^6 7^{-1} \quad\mbox{ and } \quad N_1 \in \even_{(0,6)}2^4 7^{1}\]
or
\[ N_7 \in \even_{(0,6)}2^6 7^{1} \quad\mbox{ and }\quad N_1 \in \even_{(2,4)}2^4 7^{-1}.\]
In either case the genus of $N_1$ contains a single class. In the first case the class of $N_1$ has roots.
\end{lemma}
\begin{proof}
 Since $\res(\Phi_1, \Phi_7) = 7$, Lemma~\ref{lem-useful} implies that the index $[N:N_7 \oplus N_1]$ divides $7$.
But in any case $7=|\Phi_7(1) \Phi_7(-1)|$ divides $\det N_7$ (see \eqref{eq-princ-det} and \eqref{eq-principal}).
Thus we obtain
$$
[N:N_7 \oplus N_1] =7.
$$
 Consequently, for all $p \neq 7$, $N \otimes \ZZ_p = \left(N_7 \otimes \ZZ_p \right)\oplus \left(N_1 \otimes \ZZ_p\right)$. In particular for $p = 2$.
 Using the description of $N_7$ as a twist of the principal $\Phi_7$-lattice we compute the two possibilities
 for the genus of $N_7$ (see Remark~\ref{remark-computer-genus}). \\
It remains to determine the genus of $N_1$.
 Since we have
$$N \otimes \ZZ_2 = \left(N_7 \otimes \ZZ_2\right) \oplus \left(N_1 \otimes \ZZ_2\right),$$
 the $2$-adic symbol of $N_1$ must be $2^4$.
To compute the $7$-adic symbol note that $N \otimes \ZZ_7$ is unimodular, thus Lemma \ref{lem-gluesymbol} applies.
As $(-1)$ is a non-square in $\ZZ_7$ this means that the signs $\epsilon_7$ of the $7$-modular Jordan constituents of $N_7$ and $N_1$ must be different.
 The claim that $N_1$ is unique in its genus in the first case is checked with a computer algebra system (see Remark~\ref{remark-computer-genus}).
In the second case $N_1$ is indefinite and we can use \cite[Thm.15.19]{conway-sloane-book}.
\end{proof}
Recall that
$\KtS$ (resp. $\tilde{f} \in \mbox{Aut}(\KtS)$) stands for the  K3-cover of an Enriques surface $\EnS$
(resp. for a lift of an automorphism $f\in \mbox{Aut}(\EnS)$).

\begin{proposition} \label{prop-one-in-genus}
Let $\EnS$ be an  Enriques surfaces with an automorphism  $f \in \Aut(\EnS)$ such that
\eqref{eq-f7div} holds. Then $\NS(\KtS)$ contains a primitive $\tilde{f}^*$-invariant sublattice which belongs to  the genus
 $\even_{(1,15)}2^4 7^{1}$ and $N \cap \NS(\KtS)$ contains the $\tilde{f}^*$-invariant sublattice $A_6(-2)\cong N_7 \in \even_{(0,6)}2^6 7^{1}$ primitively.
\end{proposition}
\begin{proof}
Since $F_7$ divides $p_f$, \eqref{eq-div} implies that the characteristic polynomial $p_N$ is divisible by the cyclotomic polynomial $\Phi_7$.
Moreover, after replacing $f$ by $f^k$ with $k\in \mathbb{N}$ coprime to $7$, we may assume that
 \[p_N=\Phi_7(x)\Phi_1(x)^6.\]
Now  we can apply Lemma~\ref{lemm-twocases}. The first case is impossible as then $N_1$ is contained in $\NS(\KtS) \cap N$ and contains roots (see \eqref{eq-no-roots-in-N}).
 Thus we are left with the second case.
Since $N_1 \subseteq N$  is of signature $(2,4)$ it must contain the transcendental lattice (and $f$ is semi-symplectic). Thus the orthogonal complement of $N_1$ in $H^2(\KtS,\ZZ)$ is the sought for $\tilde{f}^*$ invariant sublattice of $\NS(\KtS)$.
\end{proof}

Finally, we apply Algorithm~\ref{algo:main}  to check that 
 the condition of Prop.~\ref{prop-one-in-genus} indeed gives 
Enriques surfaces such that  \eqref{eq-f7div} holds.
\begin{proposition}\label{prop-f7gens}
If the K3 cover $\KtS$ of an Enriques surface $\EnS$ satisfies the following conditions:
\begin{anumerate}
\item
$\NS(\KtS) \in \even_{(1,15)}2^4 7^{1}$ and
\item $N \cap \NS(\KtS) \cong  A_6(-2) \in \even_{(0,6)}2^6 7^{1}$.
\end{anumerate}
then the image of $\Aut_{s}(\EnS)\rightarrow \OG(\NumS) \otimes \FF_2$ generates a group isomorphic to $\mathcal{S}_7$.
In particular, 
the Enriques surface $\EnS$ admits an automorphism $f \in \Aut(\EnS)$ such that 
the modulo-$2$ reduction $(p_f(x)  \bmod 2)$ is divisible by the polynomial
$F_{7}$.
\end{proposition}
\begin{proof} Apply Algorithm~\ref{algo:main} and \cite{gap} as in the proof of Prop.~\ref{prop:F9rho16}. 
\end{proof}



\newcommand{\ComplexNumbers}{\mathord{\mathbb C}}
\newcommand{\RealNumbers}{\mathord{\mathbb R}}
\newcommand{\GGG}{\mathord{\mathcal G}}
\newcommand{\PPP}{\mathord{\mathcal P}}
\newcommand{\RRR}{\mathord{\mathcal R}}

\newcommand{\inv}{\sp{-1}}

\newcommand{\sperp}{\sp\perp}
\newcommand{\sprime}{\sp\prime}
\newcommand{\ample}{\alpha}
\newcommand{\dual}{\sp{\vee}}
\newcommand{\inj}{\hookrightarrow}
\newcommand{\set}[2]{\{\,{#1}\mid {#2} \,\}}

\newcommand{\tensor}{\otimes}

\newcommand{\aut}{\mathrm{aut}}
\newcommand{\Isom}{\mathord{\mathrm{Isom}}}
\newcommand{\SSSS}{\mathord{\mathfrak S}}

\newtheorem{_algorithm}[theorem]{Algorithm}
\newenvironment{algorithm}{\begin{_algorithm}\rm}{\hfill \rule{3pt}{6pt}
\end{_algorithm}}

\newtheorem{_algorithm*}[theorem]{Algorithm}
\newenvironment{algorithm*}{\begin{_algorithm}\rm
}{\hfill
\end{_algorithm}}

\section{Appendix: an algorithm to calculate  generators}\label{sec:appendix}

In this appendix,
we present an algorithm to calculate a finite generating set of
the image of the natural homomorphism
from the automorphism group of an Enriques surface to
the orthogonal group of the numerical N\'eron-Severi lattice of the Enriques surface.
Our algorithm is based on Borcherds' method~\cite{borcherds1, borcherds2}
with the result in~\cite{BrandhorstShimada2019}.

\subsection{Borcherds' method}
We use the notation and terminologies in~\cite{BrandhorstShimada2019}.
In particular,
we denote by $Y$ an Enriques surface,
$\pi\colon X\to Y$ the universal covering of $Y$, and
$S_X$ and $S_Y$ the numerical N\'eron-Severi lattices of $X$ and of $Y$,
respectively
(that is, $S_X=\NS(X)$ and $S_Y=\operatorname{Num}(Y)$
in the notation of previous sections.)
Let $\PPP_X$ (resp.~$\PPP_Y$) be the positive cone of $S_X\tensor\R$ (resp.~$S_Y\tensor\R$)
containing an ample class.
Let  $N_X$ (resp.~$N_Y$) be the cone
consisting of all $x\in \PPP_X$ (resp.~all $x\in \PPP_Y$)
such that $\intf{x, [\Gamma]}\ge 0$ for any curve $\Gamma$ on $X$ (resp.~on $Y$).
We let the orthogonal group $\OG(L)$ 
of a $\Z$-lattice $L$ act on the lattice from the \emph{right}.
Suppose that $L$ is even.
A vector $r\in L$ is a \emph{$(-2)$-vector} if $\intf{r,r}=-2$.
Let $W(L)$ denote the subgroup of $\OG(L)$ 
generated by the reflections $s_r\colon x\mapsto x+\intf{x, r} r$
with respect to $(-2)$-vectors $r$ of $L$.
For a subset $A$ of $L\tensor \RealNumbers$,
we denote by $A^g$ the image of $A$ under the action of $g\in \OG(L)$
(\emph{not} the fixed locus of $g$ in $A$), and put
\[
\OG(L, A):=\set{g\in \OG(L)}{A=A^g}.
\]
We have natural homomorphisms
\[
\Aut(X)\to \OG(S_X, \PPP_X),
\quad
\Aut(Y)\to \OG(S_Y, \PPP_Y).
\]
We denote by $\aut(X)$ and $\aut(Y)$ the images of these homomorphisms.
Recall that $\Aut_{s}(Y)$ consists of the semi-symplectic automorphisms, i.e. those that act trivially on $H^0(Y,\omega_Y^{\otimes 2})$. We denote by $\Aut_{s}(X)$ the
subgroup consisting of those automorphisms acting as $\pm 1$ on $H^0(X,\Omega^2_X) \cong H^{2,0}(X)$.
The subgroups $\aut_{s}(X) \subseteq \aut(X)$ and $\aut_{s}(Y) \subseteq \aut(Y)$ are defined as the respective images.
Our goal is to calculate a finite generating set of $\aut_{s}(Y)$.
\begin{remark}
We note that $\Aut_{s}(Y)$ is of finite index in $\Aut(Y)$.
This index is one if 
the only isometries of $T_X$ that preserve 
$H^{2,0}(X)\subset T_X\tensor\ComplexNumbers$ are $\pm 1$, 
where $T_X$ is the transcendental lattice of $X$.
\end{remark}
We have the primitive embedding
\[
\pi^*\colon S_Y(2)\inj S_X,
\]
which induces $\PPP_Y\inj \PPP_X$.
We regard  $S_Y$ as a submodule of $S_X$ and $\PPP_Y$ 
as a subspace of $\PPP_X$ by $\pi^*$.
Then we have
\begin{equation}\label{eq:NYNXPY}
N_Y=N_X\cap \PPP_Y.
\end{equation}
If $\alpha\in S_Y$ is ample on $Y$, then $\pi^*(\alpha)$ is ample on $X$.
Hence we have
$N_Y\sp{\circ}=N_X\sp{\circ}\cap \PPP_Y$,
where $N_Y\sp{\circ}$ and $N_X\sp{\circ}$ are the interiors of $N_Y$ and $N_X$,
respectively.
Let $Q$ denote the orthogonal complement of the sublattice $S_Y(2)$ in $S_X$.
Since $Q$ is negative-definite, the group $\OG(Q)$ is finite.
We consider the following assumptions
for an element $g$ of $\OG(S_Y, \PPP_Y)$:
\begin{itemize}
\item[(i)]  There exists an isometry $h\in \OG(Q)$ such that
the action of $g\oplus h$ on $S_Y(2)\oplus Q$
preserves the overlattice $S_X$ of $S_Y(2)\oplus Q$ and
the action of $(g\oplus h)|S_X$ on the discriminant group
$S_X\dual /S_X$ of $S_X$ is $\pm 1$.
\item[(ii-a)] There exists an ample class   $\ample\in S_Y$ of $Y$ such that
there exist no vectors $r\in S_X$ with $\intf{r, r}=-2$
satisfying $\intf{\pi^*(\ample), r}>0$ and $\intf{\pi^*(\ample^g), r}<0$.
\item[(ii-b)] For an arbitrary ample class   $\ample\in S_Y$ of $Y$,
there exist no vectors $r\in S_X$ with $\intf{r, r}=-2$
satisfying $\intf{\pi^*(\ample), r}>0$ and $\intf{\pi^*(\ample^g), r}<0$.
\end{itemize}
\begin{proposition}\label{prop:criterion}
Let $g$ be an element of $\OG(S_Y, \PPP_Y)$.
Then $g$ is in $\aut_{s}(Y)$ if  {\rm (i)} and {\rm (ii-a)} hold.
If $g$ is in $\aut_{s}(Y)$, then {\rm (i)} and {\rm (ii-b)} hold.
\end{proposition}
\begin{proof}
An element $g$ of $\OG(S_Y, \PPP_Y)$ is in $\aut_{s}(Y)$ if and only if
there exists an element $\tilde{g}\in \aut_{s}(X)$ 
that preserves $S_Y\subset S_X$
and satisfies $\tilde{g}|S_Y=g$.
By the Torelli theorem, 
we see  that an element  $\tilde{g}\sprime$
of $\OG(S_X, \PPP_X)$ is in $\aut_{s}(X)$  if and only if the action of
$\tilde{g}\sprime$ on $S_X\dual/S_X$ is $\pm 1$ and $\tilde{g}\sprime$
preserves $N_X$.
Since $N_X$ is a standard fundamental domain of the action of $W(S_X)$ on $\PPP_X$
(see~Example~1.5 of~\cite{BrandhorstShimada2019}),
we have
\[
N_X\sp{\circ}\cap N_X^{h}\ne \emptyset
\;\;
\Longrightarrow
\;\;
N_X=N_X^{h}
\]
for any $h\in \OG(S_X, \PPP_X)$.
Therefore both   of (ii-a) and (ii-b) are equivalent to the condition
that $N_X^{\tilde{g}}=N_X$
for any  $\tilde{g}\in\OG(S_X, \PPP_X)$
satisfying $S_Y^{\tilde{g}}=S_Y$ and  $\tilde{g}|S_Y=g$.
\end{proof}
Suppose that we have a primitive embedding
\[
\iota_X\colon S_X\inj L_{26},
\]
where $L_{26}$ is an even unimodular hyperbolic lattice of rank $26$,
which is unique up to isomorphism.
(A more standard notation is $\mathrm{II}_{1, 25}$.)
Composing $\pi^*$ and $\iota_X$, we obtain
a primitive embedding
\[
\iota_Y\colon S_Y(2)\inj L_{26}.
\]
Let $\PPP_{26}$ be the positive cone of $L_{26}$
into which  $\PPP_Y$ is mapped.
We regard $S_Y$ as a primitive submodule of $L_{26}$,
and $\PPP_Y$ as a subspace of $\PPP_{26}$  by $\iota_Y$.
Recall from~\cite{BrandhorstShimada2019} that a Conway chamber is
a standard fundamental domain of the action of $W(L_{26})$ on $\PPP_{26}$.
The tessellation of $\PPP_{26}$ by Conway chambers
induces a tessellation of $\PPP_Y$ by induced chambers.
%
\begin{proposition}
The action of $\aut_{s}(Y)$ on $\PPP_Y$ preserves the tessellation of $\PPP_Y$
 by induced chambers.
\end{proposition}
\begin{proof}
Let $g$ be an element of $\aut_{s}(Y)$.
By the proof of Proposition~\ref{prop:criterion},
there exists an isometry $\tilde{g} \in \OG(S_X, \PPP_X)$
such that $S_Y^{\tilde{g}}=S_Y$, $\tilde{g}|S_Y=g$ and the action of $\tilde{g}$ on
$S_X\dual/S_X$ is $\pm 1$.
By the last condition, we see that
 $\tilde{g}$ further extends to an isometry $g_{26}\in \OG(L_{26}, \PPP_{26})$.
 Since the action of $g_{26}$ on $\PPP_{26}$ preserves the tessellation by Conway chambers,
 the action of $g$ on $\PPP_Y$ preserves the tessellation  by induced chambers.
\end{proof}
Let $L_{10}$ be an even unimodular hyperbolic lattice of rank $10$,
which is unique up to isomorphism.
In~\cite{BrandhorstShimada2019},
we have classified all primitive embeddings
of $S_Y(2)\cong L_{10}(2)$ into $L_{26}$,
and studied the tessellation of $\PPP_Y$ by induced chambers.
It turns out that,
up to the action of $\OG(L_{10})$ and $\OG(L_{26})$,
there exist exactly $17$ primitive embeddings $L_{10}(2)\inj L_{26}$,
and except for one primitive embedding named as ``\text{infty}",
the associated tessellation of $\PPP_Y$ by induced chambers has the following properties:
\begin{itemize}
\item Each induced chamber $D$ is bounded by a finite number of walls, and each wall is
defined by a $(-2)$-vector.
\item If a $(-2)$-vector $r$ defines a wall $w=D\cap (r)\sperp$ of an induced chamber $D$,
then the reflection $s_r\colon x\mapsto x+\intf{x, r} r$
into the mirror  $(r)\sperp$ maps $D$ to the induced chamber adjacent to $D$ across the wall $w$.
\end{itemize}
In particular, the tessellation of $\PPP_Y$ by induced chambers is \emph{simple} 
in the sense of~\cite{ShimadaHoles}.

\subsection{Main Algorithm.}
Suppose that the primitive embedding $\iota_Y$ is not of type ``\text{infty}".
Suppose also that we have calculated the walls of 
an induced chamber $D_0\subset \PPP_Y$ contained in $N_Y$.
\par
Before starting the main algorithm,
we calculate the finite groups $\OG(Q)$ and $\OG(S_Y, D_0)$.
We also fix an ample class $\alpha$ that is contained in the interior  of $D_0$.
In the following,
an induced chamber $D$ is expressed by an element $\tau_D\in \OG(S_Y, \PPP_Y)$
such that $D={D_0}^{\tau_D}$.
Note that $\tau_D$ is uniquely determined by $D$
up to left multiplications of elements of $\OG(S_Y, D_0)$.
\par
Then we have the following auxiliary algorithms.
\begin{algorithm}
Given an induced chamber $D$,
we can determine whether $D\subset N_Y$ or not.
Indeed, by~\eqref{eq:NYNXPY}, we have $D\subset N_Y$ if and only if
there exist no  $(-2)$-vectors $r$ of $S_X$ such that
$\intf{\pi^*(\alpha), r}>0$ and $\intf{\pi^*(\alpha^{\tau_D}), r}<0$.
The set of such $(-2)$-vectors can be calculated by
the algorithm in Section~3.3~of~\cite{ShimadaChar5}.
\end{algorithm}
Suppose that $D\subset N_Y$.
A wall $D\cap (r)\sperp$ of $D$ is said to be \emph{inner} 
if the induced chamber $D^{s_r}$ adjacent to $D$ across $D\cap (r)\sperp$
is contained in $N_Y$.
Otherwise, we say that $D\cap (r)\sperp$  is \emph{outer}.
\begin{algorithm*}\label{algo:aux2}\hfill\\[3pt]
\textbf{Input:} An embedding $S_Y(2) \hookrightarrow S_X \hookrightarrow L_{26}$, the groups $\OG(S_Y, D_0)$, $\OG(Q)$ and two induced chambers $D, D\sprime\subset N_Y$ represented by $\tau_D, \tau_{D\sprime}$. \\
\textbf{Output:}
The set $\{\gamma\in \aut_{s}(Y) \mid D\sprime=D^{\gamma}\}$.
\begin{algorithmic}[1]
\State Compute $\Isom(D, D\sprime) := \tau_D\inv \OG(S_Y, D_0) \tau_{D\sprime}$.\newline
This is the set of all isometries $g\in \OG(S_Y, \PPP_Y)$
that satisfy $D\sprime=D^{g}$.
\State Initialize $\mathcal{I} := \{\}$
\For {$g \in \Isom(D, D\sprime)$}\newline\hspace*{12pt}
Use $\OG(Q)$ and Proposition~\ref{prop:criterion} to check
 \If {$g \in \aut_{s}(Y)$}
  \State add $g$ to $\mathcal{I}$.
 \EndIf
\EndFor
\State Return $\mathcal{I}$.
\end{algorithmic}
\end{algorithm*}
Note that since both $D$ and $D\sprime$ are contained in $N_Y$, 
condition (ii-a) of Proposition~\ref{prop:criterion} is always satisfied in line 4.
%
For $D = D\sprime$, Algorithm~\ref{algo:aux2} calculates
the group
\[
\aut_{s}(Y, D):=\OG(S_Y, D)\cap \aut_{s} (Y).
\]
Two induced chambers $D$ and $D\sprime$ in $N_Y$ are said to be
\emph{$\aut_{s}(Y)$-equivalent} if there exists an element $\gamma\in \aut_{s}(Y)$ such that
$D\sprime=D^{\gamma}$.

\begin{algorithm*}\label{algo:main}\hfill\\[3pt]
\textbf{Input:} An embedding $S_Y(2) \hookrightarrow S_X \hookrightarrow L_{26}$ \newline
and an induced chamber $D_0\subset N_Y$. \\
\textbf{Output:} A list $\RRR$ of representatives of
$\aut_{s}(Y)$-equivalence classes of  
induced chambers contained in $N_Y$ and a generating set $\GGG$ of $\aut_{s}(Y)$.
\begin{algorithmic}[1]
\State Initialize $\RRR:=[D_0]$,  $\GGG:= \{\}$ and $i:= 0$.
\While{$i \leq |\RRR|$}
    \State Let $D_i$ be the $(i+1)$st element of $\RRR$.
    \State Replace $\GGG$ by $\GGG \cup \aut_{s}(Y, D_i)$.
    \State Let $\mathcal{W}$ be the set of walls of $D_i$.
    \State Compute orbit representatives of $\mathcal{W}$ under the action of $\aut_{s}(Y, D_i)$.
    \For {each representative wall $w$ of $ \mathcal{W}/\aut_{s}(Y, D_i)$}
        \State Let $r$ be the $(-2)$-vector of $S_Y$ defining the wall $w=D\cap (r)\sperp$.
        \State Let $s_r$ be the reflection $x\mapsto x+\intf{x, r}r$.
        \State Let $D_w=D_i^{s_r}$ be the induced chamber adjacent to $D_i$ across $w$.
        \State Set $\tau_{D_w}:= \tau_{D_i} s_r$.
        \If{$D_w \not\subset N_Y$}
            \State continue with the next representative wall.
        \EndIf
        \State  Set $f:=\mathrm{true}$.
        \For {each $D\in \RRR$}
        	\If{$D$ is $\aut_{s}(Y)$-equivalent to
                $D_{w}$}
            	\State Let $\gamma\in \aut_{s}(Y)$ be an element such that $D_w=D^\gamma$.
            	\State Add $\gamma$ to $\GGG$.
            	\State Replace $f$ by $\mathrm{false}$.
            	\State Break the for loop.
        	\EndIf
        \EndFor
        \If{$f=\mathrm{true}$}
        	\State Add $D_w$ to $\RRR$.
        \EndIf
    \EndFor
    \State Increment $i$.
\EndWhile
\State Return $\RRR$ and $\GGG$.
    \end{algorithmic}
\end{algorithm*}

\begin{proof}
This Algorithm is proved in the same way as the proof of
Proposition~6.3~of~\cite{shimada-algorithm}.
\end{proof}
\subsection{Examples}
The details of the following computations are available at~\cite{shimada-comp-BRS}.
%

%
\subsubsection{The Enriques surface in Proposition~\ref{prop-f7gens}}
The Picard number of the covering $K3$ surface is $16$,
and the orthogonal complement $Q$ of $S_Y(2)$ in $S_X$ is $A_6(-2)$.
Therefore $\OG(Q)$ is of order $10080$.
The $ADE$-type of $(-2)$-vectors
in the orthogonal complement $P$ of $S_Y(2)$ in $L_{26}$ is $8A_1+2 D_4$.
Hence the embedding $\iota_Y$ is of type \text{40B} in the notation
of~\cite{BrandhorstShimada2019}.
The number $|\RRR|$ of 
$\aut_{s}(Y)$-equivalence classes of induced chambers in $N_Y$
is $2$.
Let $D_0$ and $D_1$ be the representatives 
of $\aut_{s}(Y)$-equivalence classes.
For $i=0, 1$,
the group  $\aut_{s}(Y, D_i)$ is isomorphic to $\Z/2\Z\times \Z / 2\Z$
and the $40$ walls of $D_i$  are decomposed into $10$ orbits under 
the action of $\aut_{s}(Y, D_i)$.
Among the $40$ walls,  exactly $3\times 4=12$ walls are outer walls.
For each inner wall $w$, the two induced chambers containing $w$
are not $\aut_{s}(Y)$-equivalent, that is,
one is $\aut_{s}(Y)$-equivalent to $D_0$ and the other is $\aut_{s}(Y)$-equivalent to $D_1$.
\subsubsection{The Enriques surface in Proposition~\ref{prop:F9rho16}}
The Picard number of the covering $K3$ surface is $16$,
and the orthogonal complement $Q$ of $S_Y(2)$ in $S_X$ is $E_6(-2)$.
Therefore $\OG(Q)$ is of order $103680$.
The $ADE$-type of $(-2)$-vectors
in the orthogonal complement $P$ of $S_Y(2)$ in $L_{26}$ is $D_4+D_5$.
Hence the embedding $\iota_Y$ is of type \text{20A},
which means that $D_0$ is bounded by walls defined by $(-2)$-vectors
that form the dual graph of Nikulin-Kondo's type $\mathrm{V}$~\cite{KondoFinite}.
The number $|\RRR|$ of $\aut_{s}(Y)$-equivalence classes of induced chambers in $N_Y$
is $20$.
They are decomposed into the following three types.
\[
\begin{array}{ccccc}
\text{Type} & |\aut_{s}(Y, D)| &\text{outer walls} &\text{inner walls} &
\text{number} \\
\hline
\mathrm{a} & 1 & 1\times 7  & 1\times 13 & 2 \\
\mathrm{b} & 1 & 1\times 5  & 1\times 15 & 6 \\
\mathrm{c} & 2 & 1\times 2 + 2\times 2   & 1\times 2 + 2\times  6 & 12.
\end{array}
\]
For example, 
there exist twelve $\aut_{s}(Y)$-equivalence classes of type $\mathrm{c}$.
If $D$ is an induced chamber of type $\mathrm{c}$,
then $\aut_{s}(Y, D)$ is $\Z/2\Z$, and $D$ has $6$ outer walls and $14$ inner walls.
Under the action of $\aut_{s}(Y, D)$,
the $6$ outer walls are decomposed into $4$ orbits of size  $1,1,2,2$,
and the $14$ inner walls are decomposed into $8$ orbits of size $1,1, 2, \dots, 2$.
\subsubsection{The Enriques surface in Proposition~\ref{prop:F9rho18}}
The Picard number of the covering $K3$ surface is $18$,
and the orthogonal complement $Q$ of $S_Y(2)$ in $S_X$ is $A_8(-2)$.
Therefore $\OG(Q)$ is of order $725760 $.
The $ADE$-type of $(-2)$-vectors
in the orthogonal complement $P$ of $S_Y(2)$ in $L_{26}$ is $A_3+A_4$.
Hence the embedding $\iota_Y$ is of type \text{20D},
which means that $D_0$ is bounded by walls defined by $(-2)$-vectors
that form the dual graph of Nikulin-Kondo's type $\mathrm{VII}$~\cite{KondoFinite}.
The number $|\RRR|$ of $\aut_{s}(Y)$-equivalence classes of induced chambers in $N_Y$
is $1$.
The group $\aut_{s}(Y, D_0)$ is isomorphic to $\SSSS_3$,
and the $20$ walls of $D_0$ are decomposed into $6$ orbits,
each of which  consists of
\[
6\; \text{outer},\;\;
3\; \text{outer},\;\;
3\; \text{outer},\;\;
\quad
3\; \text{inner},\;\;
3\; \text{inner},\;\;
2\; \text{inner}.
\]

\end{document}